\documentclass[10pt]{amsart}
\usepackage[leqno]{amsmath}

\usepackage{amsmath}
\usepackage{amsthm}
\usepackage{amssymb}
\usepackage{amscd}
\usepackage{amsxtra}     % Use various AMS packages
\usepackage{epsfig}
\usepackage{verbatim}
\usepackage[all]{xypic}

\theoremstyle{plain} %default

\newtheorem{thm}{Theorem}[section]
\newtheorem{lem}[thm]{Lemma}
\newtheorem{prop}[thm]{Proposition}

\newtheorem{cor}[thm]{Corollary}

\theoremstyle{definition}
\newtheorem{dfn}[thm]{Definition}
\newtheorem{construct}[thm]{Construction}
\newtheorem{eg}[thm]{Example}

\newtheorem{rmk}[thm]{Remark}

\numberwithin{equation}{section}

\newcommand{\tensor}{\otimes}

 \DeclareMathOperator{\Tor}{Tor}
 \DeclareMathOperator{\Ext}{Ext}
 \DeclareMathOperator{\Hom}{Hom}

\DeclareMathOperator{\add}{add}

 \DeclareMathOperator{\Supp}{Supp}
 \DeclareMathOperator{\Spec}{Spec}
 
 \DeclareMathOperator{\Se}{S}
\DeclareMathOperator{\Reg}{R}

 \DeclareMathOperator{\pd}{pd}

 \DeclareMathOperator{\depth}{depth}

 \DeclareMathOperator{\LC}{H}

 \DeclareMathOperator{\pe}{pe}
  \DeclareMathOperator{\modu}{mod}
   \DeclareMathOperator{\MCM}{MCM}
      \DeclareMathOperator{\gl}{gl.dim}

\title{Remarks on non-commutative crepant resolutions of complete intersections}
\author{Hailong Dao}

\address{Department of Mathematics, University of Kansas,  405 Snow Hall, 1460 Jayhawk Blvd, Lawrence,
KS 66045-7523, USA} \email{hdao@math.ku.edu}

\thanks{The author is partially supported by NSF grant 0834050}



\begin{document}

\bibliographystyle{plain}

\maketitle

\begin{abstract}
We study obstructions to existence of non-commutative crepant resolutions, in the sense of Van den Bergh, over  local complete intersections.
\end{abstract}

\section{Introduction}

Let $R$ be a Gorenstein local normal domain. The following striking definition is due to Van den Bergh (see \cite{V1}, 4.1,4.2):
\begin{dfn}
Suppose that there exist a reflexive module $M$ satisfying:
\begin{enumerate}
\item $A = \Hom_R(M,M)$ is maximal Cohen-Macaulay $R$-module.
\item $A$ has finite global dimension.
\end{enumerate} 
Then $A$ is called a non-commutative crepant resolution (henceforth NCCR) of $R$. 
\end{dfn}

It has been shown that for dimension $3$ isolated, terminal singularities, the existence of projective and non-commutative crepant resolutions are equivalent (\cite{V1}). A projective crepant resolution is a desingularization $f: Y\to X=\Spec(R)$ such that $f^*\omega_X = \omega_Y$.

In this note we observe that the existence of non-commutative crepant resolutions is rather restrictive over complete intersections with small singular locus. In particular, they do not exist  for equicharacteristic, isolated hypersurface singularities  of dimension $3$ which are $\mathbb Q$-factorial, or of  even dimension at least $4$, even though commutative crepant resolutions are known to exist in such situation (albeit rarely).  This is in contrast with  known results in dimension $2$ or $3$. Our results also provide a new perspective on NCCR in the known cases and suggest how to build the  module $M$ in the definition of NCCR. We employ only homological methods over commutative rings, and they typically work over any field,  even in some mixed characteristic cases. 

We now describe the results of the paper. In Section \ref{intro} we give relevant definitions and preliminary results. We observe a connection between NCCR and a module-theoretic condition known as $\Tor$-rigidity (see Definition \ref{Tordef}). 

Section \ref{hyper} deals with hypersurface singularity. Our main result is:

\begin{thm}
Let $R$ be a local hypersurface satisfying condition $(\Reg_2)$. Assume $\hat R \cong S/(f)$ where $S$ is an  equicharacteristic or unramified  regular local ring and $f \in S$ is a regular element.  
\begin{enumerate}
\item If $\dim R=3$ and the class group of $R$ is torsion (i.e. $R$ is $\mathbb Q$-factorial), then $R$ admits no NCCR.  
\item If $R$ has isolated singularity and $\dim R$  is an even number greater than $3$, then $R$ admits no NCCR. 
\item Let $\mathcal N$ be the set of isomorphism classes of indecomposable maximal Cohen-Macaulay  modules over $R$ which are not  $\Tor$-rigid. Assume that $\mathcal N$ is not empty. Let $M$ a module such that $R|M$ and 
$ \mathcal N \subset \pe^n(M)$ (see Definition \ref{pe}) for some $n$.  If $A =\Hom_R(M,M)$ is $(\Se_3)$, then it has finite global dimension at most $\dim R+n$. In particular, if $A$ is MCM, then it is a NCCR over $R$. 
\end{enumerate} 
\end{thm}

We use this Theorem to analyze NCCRs of  simple singularities in dimension $3$ in \ref{simple}.

Section \ref{ci} deals with complete intersection singularities. Here we show that if $R$ is regular in codimension $3$ and $M$ satisfies Serre's condition $(\Se_3)$, then $\Hom_R(M,M)$  can not be a NCCR. We also study general conditions for NCCR to deform. Finally, we observe how our ideas can be applied to the characteristic $p$ situation, where they help explain why the non-commutative analogue of $F$-blowups often fail to be crepant. 

We would like to deeply thank Graham Leuschke for patiently explaining his paper to us and pointing our some errors on an early draft. Special thanks also go to Craig Huneke and Tommaso de Fernex for very helpful conversations. 

\section{Tor-rigidity and NCCR}\label{intro}
In this section we point out some connections between NCCR and $\Tor$-rigidity, a technical condition  well-known in commutative algebra. We begin with  relevant definitions and notations:

Let $R$ be a local ring and $M,N$ finite $R$-modules. Let $M^{*} := \Hom(M,R)$ be the dual of $M$. The module $M$ is
called \textit{reflexive} provided the natural map $M\to M^{**}$ is
an isomorphism. The module $M$ is called \textit{maximal
Cohen-Macaulay} (henceforth MCM) if $\depth_RM =
\dim R$. The ring $R$ is said to satisfy condition $(\Reg_n)$ if $R_p$ is regular for any $p\in \Spec(R)$ of codimension at most $n$. For a ring $A$, we will denote by $\gl A$ the global dimension of $A$. 

For a non-negative integer $n$, $M$ is said to satisfy Serre's condition $(\Se_n)$ if:
$$ \depth_{R_p}M_p \geq \min\{n,\dim(R_p)\} \ \forall p\in \Spec(R)$$

\begin{dfn}\label{Tordef}
A pair of $R$-modules $(M,N)$ is called $\Tor$-{\textit{rigid}} if for any
integer $i\geq0$, $\Tor_i^R(M,N)=0$ implies $\Tor_j^R(M,N)=0$ for
all $j\geq i$. Moreover, $M$ is $\Tor$-\textit{rigid} if for all $N$, the
pair $(M,N)$ is $\Tor$-rigid.
\end{dfn}

\begin{dfn}\label{pe}
Let $\mathcal X, \mathcal Y$ be subcategories of $\modu R$.  Let $\add \mathcal X$ denotes the set of all  direct summands of some direct sum of modules in $\mathcal X $. We define $\pe_R(\mathcal X,\mathcal Y)$, or $\pe(\mathcal X,\mathcal Y)$ to be the subcategory of $\modu R$ consisting of all modules $C$ such that there are exact sequence of either of the forms:
$$ 0 \to A \to C \to B \to 0$$
$$ 0 \to A \to B \to C \to 0$$
with $A \in \mathcal X$ and $B\in \mathcal Y$.  
For any integer $n\geq0$ we defines the subcategories $\pe^n\mathcal X$ inductively as follows: $\pe^0X =\add \mathcal X$, $\pe^{n+1}\mathcal X=\add(\pe(\add\mathcal X,\pe^n\mathcal X))$. We let $\pe^{\infty} \mathcal X = \cup_{n\geq 0} \pe^n \mathcal X$.  
\end{dfn}

We first record a useful:
\begin{lem}\label{useful}
Let $R$ be a Cohen-Macaulay local ring, $M,N$ are finitely generated $R$-modules and $n>1$ an integer. Consider the two conditions:
\begin{enumerate}
\item  $\Hom(M,N)$  is $(\Se_{n+1})$.
\item  $\Ext_R^i(M,N)=0$ for $1 \leq i \leq n-1$.
\end{enumerate}
If  $M$ is locally free in codimension $n$ and $N$ satisfies $(\Se_n)$, then (1) implies (2). If  $N$ satisfies $(\Se_{n+1})$, then (2) implies (1).
\end{lem}

\begin{proof}
The first claim is  obvious if $\dim R\leq n$, as then  $M$ is free by assumptions.  By localizing at the primes on the punctured spectrum of $R$ and using induction on dimension, we can assume that  all the modules  $\Ext_R^i(M,N)$, $1 \leq i \leq n-1$  have finite length. Take a free resolution of $P$ of $M$ and look at the first $n$ terms of $\Hom(P,N)$.  As all the cohomology of this complex are $\Ext$ -modules, the claim now follows from the Acyclicity Lemma (see \cite{BH}, Exercise 1.4.23). 

For the second claim, one again takes a free  resolution of $P$ of $M$ and look at $\Hom(P,N)$.  The vanishing of the $\Ext$ modules gives the long exact sequence: 
$$ 0 \to \Hom_R(M,N) \to N^{b_0} \to \cdots \to N^{b_{n-1}} \to B \to 0 $$

Counting depth shows that $\depth \Hom_R(M,N)_p \geq \min\{n+1, \depth(N_p)\}$ for any $p \in \Spec(R)$, which is what we want.

\end{proof}

The following result was first proved by Jothilingham (\cite{Jo}, Main Theorem and the discussion of the last Proposition).  For a more modern presentation, see \cite{Jor}. 

\begin{thm}\label{Jothi} (Jothilingham)
Let $R$ be a local ring and $M,N$ are finite $R$-modules such that  $N$ is  $\Tor$-rigid. If $\Ext^1_R(M,N)=0$ then $M^{*}\tensor_RN \cong \Hom_R(M,N)$ via the canonical map. In particular, $\Ext^1_R(N,N)=0$ if and only if $N$ is free. 
\end{thm}

\begin{cor}\label{keyCor}
Let $R$ be a  local ring satisfying condition $(\Reg_2)$ and $(\Se_3)$. Suppose that $M$ is a reflexive $R$-module giving an NCCR for $R$. Then any non-free  module in  $\add(M)$ is not $\Tor$-rigid. 
\end{cor}

\begin{proof}
Suppose there is a non-free summand $N$ of $M$ which is $\Tor$-rigid. Then $N$ is reflexive, so it is $(\Se_2)$ and also free in codimension $2$ as $R$ is $(\Reg_2)$. Moreover,  $\Hom_R(N,N)$ is $(\Se_3)$ as it is a summand of $\Hom_R(M,M)$.  Lemma \ref{useful} and Theorem \ref{Jothi} combine to imply that $N$ is free, contradiction.
\end{proof}
In the next section, we shall prove a partial converse to this Corollary for hypersurfaces. 

\begin{rmk}\label{class}
Suppose that $R$ is a local ring which is $(\Reg_2)$ and $(\Se_3)$. Then  an reflexive ideal $I$ representing a non-trivial element in the class group of $R$ must not be $\Tor$-rigid. Indeed, we have $\Hom_R(I,I) \cong R$ satisfies $(\Se_3)$.  
\end{rmk}

In general, it is a very delicate problem to decide whether a module (or a pair)  is $\Tor$-rigid. In the hypersurface case, however, there has been recent progress. We summarize the relevant results in:

\begin{thm}\label{rigid}
Let $R$ be a local hypersurface with isolated singularity. Assume $\hat R \cong S/(f)$ where $S$ is an  equicharacteristic or unramified  regular local ring and $f \in S$ is a regular element. Let $M$ be a finite $R$-module. 
\begin{enumerate}
\item If $[M]=0$ in $\overline{G}(R)_{\mathbb Q}$, the reduced Grothendieck group of finite $R$-modules with rational coefficients, then $M$ is $\Tor$-rigid.
\item If $\dim R=3$ and the class group of $R$ is torsion (i.e. $R$ is $\mathbb Q$-factorial), then $M$ is $\Tor$-rigid.
\item Assume $\dim R$  is an even number greater than $3$ and $M$ is reflexive. Then  $\Hom_R(M,M)$ is $(\Se_3)$ if and only if $M$ is free.
\item In this part we do not assume isolated singularity. Assume $R$ is a local hypersurface of dimension at least $3$. Suppose that $M$ is reflexive and  locally free on the punctured spectrum of $R$. Assume that $[M]=0$ or $[M^*]=0$ in $\overline{G}(R)_{\mathbb Q}$.  Then $\Hom_R(M,M)$ is $(\Se_3)$ if and only if $M$ is free.
\end{enumerate} 
\end{thm}

\begin{proof}
Part (1) and (2) are contained in Theorem 4.1 in \cite{Da1}). Part (3) is Corollary 4.4 in \cite{Da2}. Part (4) is  Theorem 3.4 in \cite{Da2}. 
\end{proof}

Next we shall discuss how to construct projective  resolution over an endomorphism ring. It was explained to us by Craig Huneke. The idea follows \cite{ARS}, as explained in \cite{Leu}, however we need a bit more details for our purpose.  For finite $R$-modules $M,N$ we shall write $N|M$ if $N$ is a direct summand of $M$.

\begin{construct}\label{cons}
Let $M$ be a finite $R$-module and $A=\Hom_R(M,M)$. It is well known that there is an equivalence between the categories of modules in $\add(M)$  and projective modules over $A$ via $\Hom_R(M,-)$ (see for example Lemma 4.12 in \cite{BD} or \cite{Leu}). It follows that any finite $A$-module $N$ fits into an exact sequence  $$  0 \to \Hom_R(M,N_1) \to \Hom_R(M,P_1) \to \Hom_R(M,P_0) \to N \to 0$$  

The above discussion show that when investigating projective resolutions of $A$-modules if suffices to consider modules of the form $\Hom_R(M,N)$. If $R|M$, one could build a resolution in a particularly nice way. First pick a minimal set of generators $f_1,\cdots,f_n$ of $\Hom_R(M,N)$ which includes a minimal set of generators of $\Hom_R(R,N)$. Let $\phi$ be the map $M^n \to N$ which takes $(m_1,\cdots,m_n)$ to $f_1(m_1) +\cdots + f_n(m_n)$. Clearly $\phi$ is surjective and $\Hom_R(M,\phi): A^{\oplus n} \to \Hom_R(M,N)$ is also surjective. In other words, one has the short exact sequences:

$$ 0 \to N_1 \to M^{\oplus n} \to N \to 0$$
and
$$0 \to \Hom_R(M,N_1) \to A^{\oplus n} \to \Hom_R(M,N) \to 0$$
Continuing in this fashion one can build an exact complex:  
$$\mathcal{F}:  \cdots \to M^{ n_{i+1}} \to M^{n_i} \to \cdots \to M^ {n_0} \to N \to 0$$
such that $\Hom_R(M,\mathcal{F})$ is an $A$- projective resolution of $\Hom_R(M,N)$.
\end{construct} 

\begin{cor}\label{remark1}
 $$\gl(A) \leq \sup \{\pd_A\Hom_R(M,N) | N\in \modu(R)\} + 2$$
\end{cor}

\begin{proof}
As in \ref{cons}, any $A$-module has a second syzygy of the form $\Hom_R(M,N)$. 
\end{proof}

\begin{cor}
Let $R$ be a local ring and $M$ be an $R$-module such that $R|M$ and $\gl \Hom_R(M,M)$ is finite. Then $\add(M)$ generate the Grothendieck group of $\mod(R)$.
\end{cor}

\begin{rmk}
The above construction shows, as proved  by Leuschke (Theorem 6, \cite{Leu}), that if the ring $R$ has finite CM type, and one takes $M$ to be the direct sum of all the representatives of the indecomposable  MCM  modules, then $A=\Hom_R(M,M)$ will have finite global dimension. Thus, if $A$ is MCM itself, it will be an NCCR over $R$. In dimension $2$, $A$ would be automatically MCM, so this process works very well. However, in dimension $3$ or higher, $A$ is rarely MCM, and indeed  \ref{keyCor} indicates that we have to pick the non-$\Tor$-rigid modules among the indecomposable MCMs. We will push this idea further in the next section.
\end{rmk}

Finally, we discuss some sufficient conditions for an endomorphism ring to have finite global dimension. Our key condition is similar to the concept of cluster tilting (see, for example, \cite{BIKR}). 

\begin{prop}\label{global}
Let $R$ be a local Cohen-Macaulay ring of dimension $d$ and $M$ a MCM $R$-module which is locally  free in codimension $2$ and assume that $R|M$. Let $\mathcal X(M) =\{N \in \MCM R  | \Ext_R^1(M,N)=0 \}$. If $M\in \mathcal X(M)$ and $\mathcal X(M) \subseteq  \pe^n(M)$ for some $n$ then $A=\Hom_R(M,M)$ has finite global dimension at most $n+d+2$. 
\end{prop}

\begin{proof}
First, we shall prove via induction on $n$ that for any module  $C \in \pe^n(M)$, $\pd_A\Hom_R(M,C) \leq n$. The case $n=0$ is obvious. 
Suppose we proved our claim for $n=l$. Pick $C \in  \pe^{l+1}(M) $ and we may assume $C$ fits into one of the two sequences:
$$ 0 \to A \to C \to B \to 0$$
$$ 0 \to A \to B \to C \to 0$$
with $A \in \add(M)$ and $B\in \pe^{l}(M)$. But by assumption $\Ext_R^1(M,A)=0$, so either sequence remains exact after applying $\Hom(M,-)$. By induction hypotheses we are done. 

It suffices to prove that for any $R$-module $N$, $\Hom_R(M,N)$ has finite projective dimension at most $n+d$ over $A$, see Remark \ref{remark1}. By Lemma \ref{useful} we know that $A$ is $(\Se_3)$. From Construction \ref{cons} one can build an exact sequence:

$$\mathcal{F}:  0 \to N_d \to M^{ n_{d-1}} \to \cdots \to M^ {n_0} \to N \to 0$$
which remains exact when applying $\Hom_R(M,-)$. It follows that $N_d$ is MCM and  $\Hom_R(M,N_d)$ is $(\Se_3)$. Lemma \ref{useful} tells us that $\Ext_R^1(M,N_d)=0$, so $N_d \in \mathcal X$. By the claim, $\pd_A\Hom_R(M,N_d) \leq n$, so we are done. 

\end{proof}

\begin{rmk}
If $R$ is Gorenstein and $A$ is MCM, then we only need to assume $\mathcal X \subseteq  \pe^{\infty}(M)$ to conclude that the global dimension of $A$ is exactly $d$, by \cite{V1}, proof of Lemma 4.2. 
\end{rmk}

\section{NCCR over hypersurfaces}\label{hyper}

In the case of hypersurfaces, one can say a lot more about NCCRs due to recent results on $\Tor$-rigidity. Throughout this section we assume that $R$ is an abstract hypersurface, i.e. that $\hat R \cong S/(f)$ where $S$ is a regular local ring and $f \in S$ is a regular element.

\begin{thm}\label{mainhyper}
Let $R$ be a local hypersurface satisfying condition $(\Reg_2)$. Assume that $\hat R \cong S/(f)$ where $S$ is an  equicharacteristic or unramified  regular local ring and $f \in S$ is a regular element.  
\begin{enumerate}
\item If $\dim R=3$ and the class group of $R$ is finite (i.e., $R$ is $\mathbb Q$-factorial), then $R$ admits no NCCR.  
\item If $R$ has isolated singularity and $\dim R$  is an even number greater than $3$, then $R$ admits no NCCR. 
\item Let $\mathcal N$ be the set of isomorphism classes of indecomposable maximal Cohen-Macaulay  (MCM) modules over $R$ which are not  $\Tor$-rigid. Assume that $\mathcal N$ is not empty. Let $M$ a module such that $R|M$ and 
$ \mathcal N \subset \pe^n(M)$ for some $n$.  If $A =\Hom_R(M,M)$ is $(\Se_3)$, then it has finite global dimension at most $\dim R+n$. In particular, if $A$ is MCM, then it is a NCCR over $R$. 
\end{enumerate} 
\end{thm}

\begin{proof}
Part (1) and (2) follow from  Theorems \ref{Jothi} and \ref{rigid}. 
It is left to  prove part (3). We may assume $d=\dim R\geq 3$. By Lemma \ref{useful} and Proposition \ref{global}, we just need to show that any non-free, indecomposable module in $\mathcal X(M)$ also belongs to   $\mathcal N$.  Pick such $K$  in $\mathcal X(M)$. We have  $\Ext_R^1(M,K)=0$ and $\Hom_R(M,K)$ is $(\Se_3)$. By  Theorem \ref{Jothi}, $M^*\tensor_R K$ is $(\Se_3)$. But one can embed $M^*$ into a free module:
$0 \to M^* \to G \to L \to 0$ such that $L$ is torsion. Tensoring with $K$ and using the fact that   $M^*\tensor_R K$ is $(\Se_3)$ forces $\Tor_1^R(L,K) =0$. As $K$ is $\Tor$-rigid, $\Tor_i^R(M^*,K)=0$ for all $i>0$. Since $R$ is a hypersurface and $M$ is not free, it now follows that $K$ is free (see \cite{HW1}), a contradiction.

\end{proof}

\begin{rmk}
In the situation of  part (1), being $\mathbb Q$-factorial and factorial are actually the same for $R$. This is known for the equicharacteristic case. We will prove  the unramified case in a forthcoming paper. 
\end{rmk}

\begin{rmk}
The conclusion of part (1), if we further assume that $R$ has terminal singularity and  the ground field is the complex numbers, can be explained using Van den Bergh results and standard facts of birational geometry. Namely, by Theorem 6.6.3 of \cite{V1}, $R$ has a projective crepant resolution $Y\to \Spec(R)$, which has  to be a small resolution (i.e. the fibre of the closed point has dimension  at most $1$). Then the pushforward of the hyperplane section on $Y$ can not be $\mathbb Q$-Cartier, so the class group of $R$ can not be torsion.  We thank Tomasso de Fernex for explaining this fact to us. 
\end{rmk}

\begin{rmk}
Part (2) can be proved  directly using Theorem \ref{Jothi} if one knows that over such hypersurface, any module is $\Tor$-rigid. We conjecture this to be true in \cite{Da1}. Recently, a proof of our conjecture in  the graded, equicharacteristic $0$ case using complex-analytic method by Walker, Moore, Piepmeyer and Spiroff has been announced, see \cite{MPSW}. 
\end{rmk}

\begin{eg}
There are examples of isolated hypersurface singularities in all dimensions which admits projective crepant resolutions: let $k$ be an algebraically closed field of characteristic $0$ and 
$R=k[[x_0,x_1,\cdots,x_n]]/(f)$ with $f=(x_0^l+ x_1^n+\cdots+x_n^n) $ and $l>n$ an integer such that  $l\equiv 1\mod n$  (see \cite{Lin}, Theorem A.4). Thus, our result shows that extra conditions will be needed for the equivalence of the existence of two definitions of crepant resolutions, in higher dimensions. 
\end{eg}

\begin{eg}\label{simple}
Let $k$ be an algebraically closed field of characteristic $0$. We now apply Theorem \ref{mainhyper} to  analyze some simple  singularities of dimension $3$ over $k$. They are well-known to be hypersurfaces of type $A_n,D_n, E_6, E_7$ or $E_8$. But the types $A_{2l}, E_6, E_8$ are factorial (in fact $\overline{G}(R)_{\mathbb Q}=0$), so they admit no NCCR. We now study the case $A_{2l+1} = k[[x,y,u,v]]/(xy+u^2-v^{2l+2})$. There are exactly $l+3$ indecomposable MCM modules up to isomorphism: $R$, $I=(x, u+v^l)$, $I^*$ and $l$ modules $M_1,\cdots,M_l$ of ranks $2$.

We claim that all the modules $M_i$ are $\Tor$-rigid. By Theorem \ref{rigid}, it is enough to show that each $[M_i]$ is $0$  in $\overline{G}(R)_{\mathbb Q}$. By Kn\"{o}rrer periodicity result (see \cite{K} and \cite{Yo}, Chapter 12), one could prove that fact by looking at dimension $1$, that is $R=k[[u,v]]/(u^2-v^{2l+2})$. In this case, $M_i$ is the first syzygy of the ideal $L_i=(u,v^i)$. As $R/L_i$ has finite length, $[R/L_i]=0$ in $\overline{G}(R)_{\mathbb Q}$, and so is the class of its second syzygy $M_i$. Remark \ref{class} now shows
that $I$ and $I^*$ are the only non-free indecomposable MCM modules which are not $\Tor$-rigid. But $I$ is the first syzygy of $I^*$, so $I^* \in \pe^1(M)$, where $M = R\oplus I$. Since $A = \Hom_R(M,M) = R^2\oplus I\oplus I^*$ is MCM, it gives a NCCR by part (3) of Theorem \ref{mainhyper}. Obviously one can also take $M=R\oplus I^*$.

One could also prove that $A$ is NCCR as follows. Every $M_i$  fits into an exact sequence $0 \to I \to M_i \to I^* \to 0$ (it can be shown by computing the length of $\Ext_R^1(I^*,I)$). Since $I$ is the first syzygy of $I^*$, it follows that every MCM module is in $\pe^2(M)$.  Then we again apply part (3) of  Theorem \ref{mainhyper}.
 
\end{eg}

\begin{rmk}
The NCCRs over simple singularities have been completely analyzed in \cite{BIKR} via different techniques. We will use the methods described here to study rigid and cluster tilting objects, in the sense of \textit{loc. cit.}, in a separate paper. 
\end{rmk}

\section{Obstructions to non-commutative crepant resolutions over complete intersections}\label{ci}

In this section we shall extend some results of the last section to the case of complete intersections, that is, rings whose completion are isomorphic to $S/(f_1,\cdots,f_n)$ with $S$ a regular local ring and the $(f_1,\cdots,f_n)$ from a regular sequence in $S$. Typically we would need to assume some mildly good depth conditions on the module $M$ which gives rise to a NCCR. That is because $\Tor$-rigidity has not been very well-understood in this generality.

The following observation partly generalizes  Grothendieck well-known theorem that  the classs group of a complete intersection which is $(\Reg_3)$ is trivial. It also places some serious restrictions
on NCCRs of complete intersections whose singular locus have codimension at least $4$.
\begin{prop}
Let $R$ be an excellent local complete intersection satisfying regularity condition $(\Reg_3)$. Suppose that $M$ satisfies $(\Se_3)$. Then $\Hom(M,M)$ is  $(\Se_4)$ if and only if $M$ is free.  
\end{prop}

\begin{proof}
Since $R$ is excellent we can complete without affecting the issues. So we may assume $R$ is $S$ modulo a regular sequence, where $S$ is a complete regular local ring. By Lemma \ref{useful} $\Ext_R^1(M,M) = \Ext_R^2(M,M)=0$. The desired result follows from Proposition 2.5 of \cite{Jor}. We give a quick explanation for completeness. Since $R$ is now complete, we can lift $M$ to a module $N$ over $S$ (see \cite{ADS}). Since over a regular local ring every module is $\Tor$-rigid (\cite{Lich}), we know that so is $M$ as $R$-module. Theorem \ref{Jothi} finishes the proof. 
\end{proof}

\begin{cor}\label{ciNCCR}
Let $R$ be an excellent local complete intersection satisfying condition $(\Reg_3)$. Suppose that $M$ is a reflexive $R$-module such that $A =\Hom_R(M,M)$ is a NCCR. Then $M$ can not satisfy $(\Se_3)$. 

\end{cor}

\begin{rmk}
In all known examples of NCCR, the module $M$ is actually MCM. The above Corollary  shows that this can not happen when $R$ is an excellent local complete intersection satisfying condition $(\Reg_3)$.
 \end{rmk}
Inspired by the above result, we shall study the issue of deforming NCCR.  We first prove a few lemmas, which should be known, but we can not find a convenient reference:

\begin{lem}\label{lengthLC}
Let $(S,m)$ be a excellent local ring and $M$ a finite $S$-module satisfying $(\Se_n)$. Then the local cohomology module
$\LC_m^n(N)$ has finite length. 
\end{lem}

\begin{proof}
We may complete and assume that $S$ is the homomorphic image of a regular local ring $(T,\mathfrak m)$. Let $d= \dim T$. Local duality over $T$ says that $\LC_m^n(M) = \LC_{\mathfrak m}^n(M)$ has finite length if and only if $\Ext_T^{d-n} (M,T)$ has finite length. Localize at any non-maximal $q \in \Spec (T) \cap \Supp(M)$. The module  $\Ext_T^{d-n} (M,T)_q \cong \Ext_{T_q}^{d-n} (M_q,T_q) $ is dual to $\LC_{qT_q}^{\dim T_q-d+n} (M_q)$ which is $0$ since $M$ is  $(\Se_n)$. 
\end{proof}

\begin{lem}\label{liftS_n}
Let $(S,m)$ be an excellent local ring and $N$ be a $S$-module. Suppose $f\in m$ is a regular element on $S$ and $N$. Let $R=S/(f)$ and $M=N/(f)$. Let $n$ be an integer. 
\begin{enumerate}
\item If $M$ is free in codimension $n$, then so is $N$.
\item If $M$ satisfies $(\Se_n)$, then so does $N$.
\item If $K$ is an $S$-module such that $f$ is $K$-regular and  $\Ext_R^1(M,K/(f))=0$ then $\Ext_S^1(N,K)=0$ and $\Hom_R(M,K/(f)) \cong \Hom_S(N,K)/(f)$.
\end{enumerate}
\end{lem}

\begin{proof}
Let $p \in \Spec(S)$ of codimension $n$. If $f \in p$, then $M_p$ is free, and (1) follows by Nakayama's Lemma. If $f \notin p$, one can choose a minimal prime $q$ of $(f,p)$. Then $q$ has codimension $n$ in $R$, so $M_q$ and thus $N_q$ is free. Since $p \subset q$, $N_p$ is free as well.   

For (2), let $V(f) = \{p\in \Spec(S)| f\in p\}$ and $U = \{p\in \Spec(S)|\depth_{R_p}M_p \geq \min\{n,\dim(R_p)\} \}$. It is standard that $U$ is open in $\Spec(S)$ (for example, see \cite{FOV}, 3.3.9). Since $S$ is local, it is enough to show that $V(f) \subset U$. We now proceed by induction on $n$. Suppose  $n=1$, we first prove that $\depth N\geq 1$. The long exact sequence of local cohomology coming from  \[ 0 \to N \xrightarrow{f}  N \to M \to 0 \] and Nakayama shows that $\LC_m^0(N)=0$. Our argument localizes, so $V(f) \subset U$, as desired. Suppose we already know that $N$ is $(\Se_{n-1})$ and $M$ is $(\Se_n)$. Again, it suffices to prove $\depth N\geq n$. By  Lemma \ref{lengthLC} we know that
 $\LC_m^{n-1}(N)$ has finite length. But using the fact that $\depth M\geq n$ and the long exact sequence of local cohomology one gets:
 \[ \cdots \to \LC_m^{n-1}(N) \xrightarrow{f}  \LC_m^{n-1}(N)  \to 0 \]
 which forces $\LC_m^n(N) =0$, which gives what we want.
 
 We now prove (3). Apply $\Hom_S(N,-)$ to the short exact sequence 
 \[ 0 \to K \xrightarrow{f}  K \to K/(f) \to 0 \]
 we get:
 \[ 0 \to \Hom_S(N,K) \xrightarrow{f}  \Hom_S(N,K)  \to \Hom_S(N,K/(f)) \to \Ext_S^1(N,K) \xrightarrow{f} \Ext_S^1(N,K) \to 0 \] 
 
 Nakayama's Lemma  provides the desired conclusions. 
\end{proof}

\begin{thm}
Let $S$ be an excellent local ring and $N$ be a $S$-module. Suppose $f$ is a regular element on $S$ and $N$. Let $R=S/(f)$ and $M=N/(f)$.  Assume $M$ is $(\Se_3)$ and  free in codimension $2$ as an $R$-module. If  $A =\Hom_R(M,M)$ is $(\Se_3)$ and has finite global dimension, then so is $B =\Hom_S(N,N)$. 
\end{thm}

\begin{proof}
We first note that $\Ext_R^1(M,M) =0$ by \ref{useful}. Now Lemma \ref{liftS_n} shows that $B$ is $(\Se_3)$. As in the proof of \ref{mainhyper} it is enough to prove that $\pd_B\Hom_S(N,K)$ is finite for any MCM $S$-module $K$ such that $\Hom_S(N,K)$ is $(\Se_3)$. But then $\Ext_S^1(N,K)=0$ and $\Hom_S(N,K/(f)) \cong \Hom_R(M,K/(f)) \cong \Hom_S(N,K)/(f)$. Now it is enough, by Nakayama, to show that $\pd_B\Hom_S(N,K/(f))$ is finite.  But by assumption $\Hom_S(N,K/(f))$ has a finite resolution by projective $A$-modules. Since each projective $A$-module has finite projective dimension over $B$ (in fact $\pd_BA=1$),  we are done.
\end{proof}

\begin{cor}\label{liftNCCR}
Let $S$ be an complete local ring and $f$ is a regular element on $S$.  Let $R=S/(f)$. Suppose that $R$ admits a NCCR 
$A =\Hom_R(M,M)$ such  that $M$ is $(\Se_4)$ and  free in codimension $3$. Then there is a lifting  $N$ of $M$ to $S$ such that $B =\Hom_S(N,N)$ is a NCCR over $S$.  
\end{cor}

\begin{rmk}
The above Corollary gives another proof of \ref{ciNCCR}.
\end{rmk}

Finally, we mention that one could use the  ideas in this paper to explain failure of certain non-commutative resolutions to be crepant in the positive characteristic case. Let $R$ be a local ring of characteristic $p$. Let $^eR$ denotes $R$ as a module over itself via the $e$-th  power of the Frobenius homomorphism.  Recently, the  module $A=\Hom_R(^eR, ^eR)$ has been shown to have finite global dimension (so it is a non-commutative resolution) in some cases (\cite{TY}). It is known that over complete intersections, $^eR$ is $\Tor$-rigid (see \cite{AM}). Hence the following result is straightforward application of Corollary \ref{keyCor} (compare with Section 6 in \cite{TY}): 

\begin{cor}
Let $R$ be a local complete intersection of characteristic $p$  such that $R$ is $(\Reg_2)$ and $e$ any integer. Then $A=\Hom_R(^eR, ^eR)$ is not a NCCR (it will not be ``crepant"). 
\end{cor}

\end{document}